\documentclass[reqno]{amsart}
\usepackage{latexsym,a4,mathrsfs,amsthm,amsmath,amssymb}
\usepackage{graphics}
\usepackage{epsfig}
\usepackage[english]{babel}
\usepackage{verbatim}
\usepackage{fancyhdr}
\usepackage{ifpdf}
\usepackage{dsfont}

\ifpdf
\usepackage{hyperref}
\else
\usepackage[hypertex]{hyperref}
\hypersetup{backref, colorlinks=true, bookmarks}
\fi

\newtheorem{theorem}{Theorem}[section]
\newtheorem{lemma}[theorem]{Lemma}

\newtheorem{definition}[theorem]{Definition}

\newtheorem{corollary}{Corollary}[theorem]
\theoremstyle{definition}

\numberwithin{equation}{section}
\newcommand{\R}{\mathbb{R}}
\newcommand{\C}{\mathbb{C}}

\newcommand{\Z}{\mathbb{Z}}
\newcommand{\N}{\mathbb{N}}
\newcommand{\eps}{\epsilon}
\newcommand{\la}{\lambda}
\newcommand{\Log}{\mathrm{Log}}

\newcommand{\norm}[1]{\| #1 \|}

\newcommand{\set}[1]{\left\{#1\right\}}
\newcommand{\E}[1]{\mathbb{E}\left[#1\right]}
\newcommand{\Pb}[1]{\mathbb{P}\left[#1\right]}
\newcommand{\Et}[1]{\mathbb{E}_\theta\left[#1\right]}
\newcommand{\Pt}[1]{\mathbb{P}_\theta\left[#1\right]}

\newcommand{\intpart}[1]{\ensuremath{\left[ #1 \right]}}
\newcommand{\fracpart}[1]{\ensuremath{\left\{ #1 \right\}}}
\renewcommand{\Re}{\mathrm{Re}}
\renewcommand{\Im}{\mathrm{Im}}

\newcommand{\scalar}[1]{\left< #1 \right>}
\newcommand{\scalarHS}[1]{\left< #1 \right>_{H.S.}}
\newcommand{\one}{\mathds{1}}
\newcommand{\ts}{\mathbf{t}}

\begin{document}

\title[Central limit theorem for multiplicative class functions]{Central limit theorem for multiplicative class functions on the symmetric group}

\author{Dirk Zeindler$^1$}
\thanks{$^1$support by the SNF}


\address{%
Department of Mathematics, University York, York, YO10 5DD, UK
}
\email{dz549@york.ac.uk}

\maketitle

\begin{abstract}
Hambly, Keevash, O'Connell and Stark have proven a central limit theorem
for the characteristic polynomial of a permutation matrix with respect to the uniform measure on the symmetric group.
We generalize this result in several ways.
We prove here a central limit theorem for multiplicative class functions on symmetric group with respect to the Ewens measure
and compute the covariance of the real and the imaginary part in the limit.
We also estimate the rate of convergence with the Wasserstein distance.
\end{abstract}

\tableofcontents
%
\newpage
\section{Introduction}
The study of random matrices has gained importance in many areas of mathematics and physics, for example in
nuclear physics, infinite dimensional integrable systems and large$-n$ representation theory.
Random matrix theory (RMT) was in the recent years also of big interest in number theory since the study of the spectrum of
the characteristic polynomial of a random matrix in a compact Lie group was central in obtaining conjectures.
%
A good example to illustrate this is the paper of Keating and Snaith \cite{snaith}.
They conjectured that the Riemann zeta function on the critical line could be
modeled by the characteristic polynomial of a random unitary matrix considered on the unit circle.
One of the results in \cite{snaith} is
\begin{theorem}
 Let $x$ be a fixed complex number with $|x|=1$ and $g_n$
be a unitary matrix chosen at random with respect to the Haar measure. Then
\begin{align}
 \frac{\Log\Bigl(\det(I_n-xg_n)\Bigr)}{\sqrt{\frac{1}{2} \log(n)}}
\xrightarrow{d} \mathcal{N}_1 + i\mathcal{N}_2 \text{ for }n\to\infty
\end{align}
and $\mathcal{N}_1,\mathcal{N}_2$ independent, normal distributed random variables.
\end{theorem}
Constin and Lebowitz have proven 10 years earlier in \cite{PhysRevLett.75.69} a weaker version of this theorem.
They showed that

\begin{align}
  \Im\left(   \frac{\Log\Bigl(\det(I_n-xg_n)\Bigr)}{\sqrt{\frac{1}{2} \log(n)}}   \right) \stackrel{d}{\to} \mathcal{N}_1.
\end{align}

They conjectured that the same is true for the real part and that the imaginary part and the real part are independent in the limit,
but haven't been able to prove this.\\


The situation for the characteristic polynomial of permutation matrices is similar.\\
A permutation matrix is a unitary matrix of the form $(\delta_{i,\sigma(j)})_{1\leq i,j\leq n}$ with $\sigma\in S_n$
and $S_n$ the symmetric group. It is easy to see that the permutation matrices form a group isomorphic to $S_n$.
We call for simplicity both groups $S_n$ and use this identification without mentioning it explicitly.
%
%
%
The characteristic polynomial $Z_n(x)$ of a permutation matrix is defined as
\begin{align}
\label{eq_def_Z_n_simple}
Z_n(x)=Z_n(x)(\sigma):=\det(I-x \sigma) \text{ with } x\in\C, \sigma\in S_n.
\end{align}
Hambly, Keevash, O'Connell and Stark have proven in \cite{HKOS}
\begin{theorem}[B.M.Hambly, P.Keevash, N.O'Connell and D.Stark]
\label{thm_HKOS}
Let $g\in S_n$ be chosen uniformly at random and $x$ be a fixed complex number with $|x| = 1$, not a root of unity and of finite type (see Definition~\ref{def_finite_type}).
Then
\begin{align}
\Re\left(\frac{\Log\bigl(Z_n(x)\bigr)}{\sqrt{\frac{\pi^2}{12}\log(n)}} \right) \xrightarrow{d} \mathcal{N}_a,
\qquad
\Im\left(\frac{\Log\bigl(Z_n(x)\bigr)}{\sqrt{\frac{\pi^2}{12}\log(n)}} \right) \xrightarrow{d} \mathcal{N}_b
\end{align}
with $\mathcal{N}_a,\mathcal{N}_b$ standard normal distributed random variables.
\end{theorem}
As for unitary matrices, it is natural to ask if $\mathcal{N}_a$ and $\mathcal{N}_b$ are independent.
This question is not considered in \cite{HKOS} and was the main motivation for this paper.
We will see in Corollary~\ref{cor_extension_HKOS} that $\mathcal{N}_a$ and $\mathcal{N}_b$ are indeed independent.
But we can show here much more. The main result of this paper is Theorem~\ref{thm_main_result}, which is an extension of Theorem~\ref{thm_HKOS}
with three important differences.
These differences are
\begin{itemize}
 \item We compute the covariance of the real and the imaginary part in the limit.
 \item We endow $S_n$ with the Ewens measure (see Definition~\ref{def_ewens_measure}), which is a generalization of the uniform measure.
 \item We consider more general class functions on $S_n$, the so called multiplicative class functions $W^n(f)$ (see Definition~\ref{def_Wn(f)}).
\end{itemize}

We have introduced multiplicative class functions $W^n(f)$ in \cite{associated_class} as generalization of $Z_n(x)$
and studied there the asymptotic behavior of their moments with respect to the uniform measure on $S_n$.\\
The main idea there is to write down the generating function with a combinatorial argument
and to use function theory to extract the asymptotic behavior.

The structure of this paper is as follows:
we introduce in Section~\ref{sec_def_and_limit_thm} multiplicative class functions $W^n(f)$ and state the main theorem of this paper.
We do in Section~\ref{sec_preparations} some preparations and state in Section~\ref{sec_weak_limit_thm} an auxiliary central limit theorem.
We then prove in Section~\ref{sec_poof_of_central_limit_for_Wn} the main Theorem~\ref{thm_main_result}.
In Section~\ref{sec_wasserstein} we then estimate the convergence rate of the probability measures
in the main theorem with the Wasserstein distance.

\section{Definition and main theorem}
\label{sec_def_and_limit_thm}

We introduce in Section~\ref{sec_Sn_introd} the Ewens measure and some well know functions on $S_n$.
In Section~\ref{sec_def_of_Wn}, we give an alternative expression for $Z_n(x)$
and use this expression to introduce the multiplicative class functions $W^n(f)(x)$.
We then state in Section~\ref{sec_main_thm} the main theorem of this paper.

\subsection{The symmetric group $S_n$}
\label{sec_Sn_introd}

All functions on $S_n$  in this paper are invariant under conjugation ($u(hgh^{-1}) =u(g)$), i.e. they are class functions. It is therefore natural to take a look at the conjugation classes of $S_n$. These can be parameterize with partitions.
\begin{definition}
 \label{def_part}
A partition $\la$ is a sequence of nonnegative integers $\la_1 \ge \la_2 \ge \cdots$.
The size of the partition is $|\lambda|:= \sum_m \lambda_m$ and the length of $\ell(\la)$ is the largest $\ell$ such that $\la_\ell \neq 0$. We call $\la$ a partition of $n$ if $|\la| = n$, and denote this by $\la\vdash n$.
\end{definition}
Let $\sigma\in S_n$ be arbitrary
and write $\sigma = \sigma_1 \sigma_2\cdots \sigma_\ell\in S_n$ with $\sigma_i$ disjoint cycles of length $\la_i$.
Since disjoint cycles commute, we can assume that $\la_1 \geq \la_2 \geq \cdots \geq \la_\ell$.
We call the partition $\la = ( \la_1,\la_2,\dots,\la_\ell) $ the \emph{cycle type} of $\sigma$.
It remains to show that two elements of $S_n$ are conjugated if and only if they have the same cycle type. Since this is well known, we omit the proof here and refer to \cite[chapter 39]{bump}.\\
We introduce as next the cycles counts of a given length.
\begin{definition}
\label{def_cm_cycle}
Let $\sigma\in S_n$ be given with cycle-type $\la=(\la_1,\cdots,\la_\ell)$. We define
\begin{align}
C_m:=C_m^{(n)}:=C_m^{(n)}(\sigma):=\# \set{i : 1\leq i\leq l  \text{ and } \la_i = m}.
\end{align}
\end{definition}

The functions $C_m^{(n)}(\sigma)$ depends only on the cycle-type of $\sigma$ and are thus class functions on $S_n$.
It is clear that the cycle type $\la$ of $\sigma\in S_n$ is uniquely determined by the values of $C_1^{(n)} ,\cdots,C_n^{(n)}$.
We therefore can work with the functions $C_m^{(n)}$ or with the cycle type $\la$.
We prefer here to use $C_m^{(n)}$.

The most natural measure on $S_n$ is the uniform measure (i.e $\Pb{A}=\frac{|A|}{n!}$),
%
but there exists of course other important measures on $S_n$ than the uniform measure. One of these measures is the
\emph{Ewens measure}, appearing in population genetics (see \cite{MR0325177}). The Ewens measure is a generalization of the uniform measure and has an additional weight depending on the total number of cycles.
\begin{definition}
  Let $\theta>0$. We then set for $\sigma\in S_n$ with cycle type $\la$
  \begin{align}
    \Pt{\sigma}:=  \frac{\theta^{\ell(\la)} }{\theta(\theta+1)\cdots (n+\theta-1)}.
  \end{align}
  The measure $\Pt{.}$ is called the Ewens measure with parameter $\theta$.
\end{definition}
The uniform measure is the special case $\theta =1$.

Many results about the Ewens measure can be found in the book \cite[chapter 4]{barbour}, for instance
\begin{lemma}
\label{def_ewens_measure}
Let $\theta>0$ be given.
The distribution of $(C_1,C_2,\cdots,C_n)$ with respect to
the \emph{Ewens measure} on $S_n$ is given by
\begin{align}
\label{eq_distribution_ewens}
\Pt{(C_1=c_1,\cdots,C_n=c_n)}
=
\frac{1}{\binom{\theta+n-1}{n}}
\prod_{m=1}^n \frac{1}{c_m!} \left( \frac{\theta}{m}\right)^{c_m} \one \left(n=\sum_{m=1}^n m c_m\right)
\end{align}
and the expectation of $C_{m}^{(n)}$ is
\begin{align}
 \E{C_{m}^{(n)}} = \frac{\theta}{m}  \frac{\binom{\theta+n-m-1}{n-m}}{ \binom{\theta+n-1}{n}} \one(m \leq n).
\label{eq_expec_Cm}
\end{align}
\end{lemma}
%

\subsection{Definition of $W^n(f)$}
\label{sec_def_of_Wn}

We now come multiplicative class functions.
We first give a explicit expression for $Z_n(x)$ using the cycle counts $C_m^{(n)}$, and then
use this expression to introduce multiplicative class functions.

Let $\sigma\in S_n$ be given, then
\begin{align}
Z_n(x)
 &=
 \prod_{m=1}^n (1-x^m)^{C_m^{(n)}}.
 \label{eq_lem_Z_n_with_cycle}
\end{align}
The proof of this equation is straightforward.
One first has to take a look a the case when $\sigma$ is a cycle and then take a look at the general case.
More details can be found in \cite{EJP2010-34}.
\begin{definition}
\label{def_Wn(f)}
Let $x\in\C$ be complex number and $f:S^1 \to \C$ be a real analytic function with $S^1:= \set{z\in\C: |z| =1}$. We then define the multiplicative
class function associated to the function $f$ as
\begin{align}
 W^n(f) = W^n(f)(x)(\sigma):= \prod_{m=1}^{n} f(x^{m})^{C_m^{(n)}}.
 \label{eq_def_W_n_with_part}
\end{align}
For brevity, we simply call this a multiplicative class function.
We also set
\begin{align}
w^n(f)(x) = \log\left(W^n(f)(x)\right):= \sum_{m=1}^n C_m^{(n)} \log\bigl(f(x^m)\bigr)
\end{align}
with $\log$ the principal branch of logarithm and $\log(-y):= \log(y)+i\pi$ for $y\in\R_{>0}$ and $\log(0)=\infty$.
\end{definition}
It is clear that $W^n(f)$ and $w_n(f)$ are class functions on $S_n$ and that the characteristic polynomial is the special case $f(x)=1-x$.

Multiplicative class functions $W^n(f)$ have been introduced in \cite{associated_class} as generalization of the characteristic polynomial $Z_n(x)$. The motivation was that one can compute the asymptotic behaviour of the moments of $Z_n(x)$ and of $W^n(f)$ using the same method with only minor changes.
We will see in Section~\ref{sec_poof_of_central_limit_for_Wn} that this is here also the case.

In contrast, the extension to the Ewens measure is much more laborious. We will see in the proof of Lemma~\ref{lem_help_asyp_behavior} that the case $0<\theta<1$ causes much more work than the uniform measure.

\subsection{The main theorem}
\label{sec_main_thm}

We state in this section the main result of this paper. For this we need some small preparations.
\begin{definition}
\label{def_mahler_measure}
Let $f:S^1 \to \C$ be a real analytic function, $x\in S^1$ be arbitrary but fixed.
If $x$ is a root of unity of order $p$, i.e. $x^p=1$ and $p$ minimal, we set
\begin{align}
  m(f)(x)
  =
  \frac{1}{p}\sum_{m=1}^p \log\bigl(f(x^m)\bigr).
  \label{eq_def_mahler_measure_1}
\end{align}
If $x$ is not a root of unity, we set
\begin{align}
m(f)(x)
=
\int_{0}^1 \log\bigl(f(e^{2 \pi i s})\bigr) \ ds.
\label{eq_def_mahler_measure_2}
\end{align}
If the sum in \eqref{eq_def_mahler_measure_1} and the integral in \eqref{eq_def_mahler_measure_2} respectively does not exists, we set $m(f)(x):=\infty$.
\end{definition}

The integral in \eqref{eq_def_mahler_measure_2} exists for each real analytic $f\neq 0$ since the zeros of $f$ are isolated and
\begin{align}
   \log\bigl(f(e^{2 \pi i s})\bigr) \sim K \log(s-s_0) \ \text{ for } \ s\to s_0
\end{align}
for $s_0$ a zero of $f$. Thus $m(f)(x)= \infty$ if and only if $f\equiv 0$, or $x$ is a root of unity of order $p$ and $f(x^m)=0$ for some $1\leq m \leq p$.

One could rewrite the integral in \eqref{eq_def_mahler_measure_2} as
$\int_{S^1} \log\bigl(f(\varphi)\bigr) \ d\varphi$ with $d\varphi$ the uniform measure on $S^1$.
We have not used this because this can be easily confused  with the complex integral
$\frac{1}{2\pi i }\int_{S^1} \log\bigl(f(z)\bigr) \ dz = \int_{0}^1 \log\bigl(f(e^{2 \pi i s})\bigr)  e^{2 \pi i s}\ ds$.

We now come to the main theorem. We distinguish the cases $x$ a root of unit and $x$ not a root of unity.
For $x$ not a root of unity we assume, as in Theorem~\ref{thm_HKOS}, that $x$ of finite type (see Definition~\ref{def_finite_type}).
This condition is essential in our proof. We postpone the definition of finite type to the end of Section~\ref{sec_uniform_dist_seq} since we can illustrate there why we need this assumption.
%
\begin{theorem}
\label{thm_main_result}
Let $f:S^1\to \C$ be real analytic, $x\in S^1$ and $S_n$ be endowed with Ewens measure.
\begin{itemize}
\item
If $x$ is not a root of unity and of finite type, and all zeros of $f$ are roots of the unity, then
\begin{align}
       	 \frac{w^n(f)(x)}{\sqrt{\log(n)}}-\theta\sqrt{\log(n)} m(f)(x)  \xrightarrow{d} \mathcal{N}
         \label{eq_thm_weak_not_root_of_unity}
    \end{align}
with $\mathcal{N}$ a complex normal distributed random variable.
The covariance of the real and the imaginary part of $\mathcal{N}$ is given by
      \begin{align}
         \frac{\theta}{2}
         \Im\left( \int_{0}^1 \log^2\bigl(f(e^{2 \pi i s})\bigr) \ ds \right).
	 \label{eq_independence_not_root_of_unity}
       \end{align}
The real and the imaginary part of $\mathcal{N}$ are independent if and only if the covariance is equal to $0$.\\
\item
If $x$ is a root of unity of order $p$ and $f(x^m) \neq 0$ for all $1 \leq m \leq p$ then
     \begin{align}
        \frac{w^n(f)(x)}{\sqrt{\log(n)}}-\theta\sqrt{\log(n)} m(f)   \xrightarrow{d} \mathcal{N}
        \label{eq_thm_weak_root_of_unity}
      \end{align}
with $\mathcal{N}$ a complex normal distributed random variable.
The covariance of the real and the imaginary part of $\mathcal{N}$ is given by
	\begin{align}
          \frac{\theta}{2}
	  \Im\left(\frac{1}{p} \sum_{m=1}^{p} \log^2\bigl(f(x^m)\bigr) \right).
	  \label{eq_independence_root_of_unity}
    \end{align}
The real and the imaginary part of $\mathcal{N}$ are independent if and only if the covariance is equal to $0$.
\end{itemize}
\end{theorem}
As promised in the beginning, we now can prove
\begin{corollary}
\label{cor_extension_HKOS}
The random variables $\mathcal{N}_a$ and $\mathcal{N}_b$ in Theorem~\ref{thm_HKOS} are independent.
\end{corollary}
%
\begin{proof}
We know that $Z_n(x) = W^n(1-x)$ and can therefore can apply Theorem~\ref{thm_main_result}.
The independence thus follows if we can show that the expression in \eqref{eq_independence_not_root_of_unity} is $0$ for $f(x)=1-x$.
A simple computation shows $\log(1-e^{2 \pi i s}) = \log|1-e^{2 \pi i s}| -i\pi s$
for $s\in [-\frac{1}{2},\frac{1}{2}]$.
We get
\begin{align*}
\frac{1}{2}\Im\left( \int_{0}^1 \log^2(1-e^{2 \pi i s}) \ ds \right)
=&
\int_{-1/2}^{1/2} \Re\bigl(\log(1-e^{2 \pi i s})\bigr) \Im\bigl(\log(1-e^{2 \pi i s})\bigr)\ ds\\
=&
\int_{-1/2}^{1/2} -s\pi \log|1-e^{2 \pi i s}| \ ds
=
0.
\end{align*}
The last integral is $0$ since the integrand is odd.
\end{proof}

\section{Preliminaries}
\label{sec_preparations}

We present in this section a collection of well known definitions and results, which we need for the proof of Theorem~\ref{thm_main_result}.

\subsection{Asymptotic behavior of $C_m^{(n)}$ and the Feller-coupling}
\label{sec_cycle_counts}
%
We follow in this section the book~\cite{barbour}.

\begin{lemma}
\label{lem_cm_rw_ym}
Let $S_n$ be endowed with Ewens measure with parameter $\theta>0$. For each $m\in\N$,
the random variables $C^{(n)}_{m}$ converge as $n\to\infty$  in distribution to
a Poisson distributed random variable $Y_m=Y_{m,\theta}$ with $\E{Y_m}=\frac{\theta}{m}$.
In fact, we have for all $b\in\N$
\begin{align}
 (C_{1}^{(n)},C_{2}^{(n)},\cdots,C^{(n)}_{b}) \xrightarrow{d} (Y_1,Y_2,\cdots,Y_b) \qquad (n\to\infty),
\end{align}
with all $Y_m$ independent.
\end{lemma}
\begin{proof}
  See \cite[Section~4]{barbour}.
\end{proof}
One of the problems of convergence in distribution of a sequence $(X_n)_{n\in\N}$ is that usually
all $X_n$ are defined on different probability spaces and it is thus very difficult to compare them directly. This is the case for $C_m^{(n)}$ and $C_m^{(n+1)}$.
Fortunately, the Feller coupling constructs for each $\theta>0$ a probability space and new random variables
$C_m^{(n)}$ and $Y_m$ on this space, which have the same distributions as the $C_m^{(n)}$ and $Y_m$ above
and can easily be compared.

The construction works as follows:
Let $\xi:=(1\xi_2\xi_3\xi_4\xi_5\cdots)$ be a sequence of independent Bernoulli random variables
with $\E{\xi_m}=\frac{\theta}{\theta+(m-1)}$.
An $m-$spacing is a sequence of $m-1$ consecutive zeroes in $\xi$ or its truncations:
$$1\underbrace{0\cdots0}_{m-1\text{ times}}1.$$
\begin{definition}
\label{def_cm_feller}
Let $C_m^{(n)}(\xi)$ be the number of m-spacings in $1\xi_2\cdots \xi_n 1$.
We define $Y_m(\xi)$ to be the number of m-spacings in the whole sequence $\xi$.
\end{definition}
\begin{theorem}
\label{thm_feller_conv}
We have
          \begin{itemize}
            \item The above-constructed $C_m^{(n)}(\xi)$ have the same distribution as the $C_m^{(n)}(\lambda)$ in Definition~\ref{def_cm_cycle}.
            \item $Y_m(\xi)$ is a.s. finite and Poisson distributed with $\E{Y_m(\xi)}=\frac{\theta}{m}$.
            \item All $Y_m(\xi)$ are independent.
            \item We have 
            \begin{align}
              \Et{\left|C_m^{(n)}(\xi)-Y_m(\xi)\right|}\leq \begin{cases}
                                                             \frac{\theta(\theta+1)}{\theta+n}   & \text{ if }\theta \geq 1,\\
							     \frac{\theta(\theta+1)}{\theta+n-m} & \text{ if } 0<\theta<1.
                                                            \end{cases}
            \end{align}
            \item For any fixed $b\in\N$,
            $$\Pb{(C_1^{(n)}(\xi),\cdots, C_b^{(n)}(\xi))\neq(Y_1(\xi),\cdots ,Y_b(\xi))}\to 0\ (n\to\infty).$$
          \end{itemize}
\end{theorem}
\begin{proof}
See \cite[Chapter 4]{barbour} and \cite[Theorem 2]{MR1177897}.
\end{proof}
We write $C_m^{(n)}$ and $Y_m$ for both sets of random variables and do not distinguish them anymore.\\
\subsection{Elementary analysis}
\label{sec_elementary_analysis}
We give here some simple and well known results from analysis.
We state them without further comments.
\begin{lemma}[Abel's partial summation]
\label{lem_abel_summation}
Let $a_1,\cdots,a_n,b_1,\cdots,b_n$ be complex numbers. Then
\begin{align}
\sum_{m=1}^n a_m b_m = A_n b_n + \sum_{m=1}^{n-1} A_m (b_{m+1}-b_m)
\label{eq_abel_summation}
\end{align}
with $A_m:= \sum_{k=1}^m a_k$ for $m\geq 1$ and $A_0:=0$.
%
\end{lemma}
%
We can use this to prove
%
%
\begin{lemma}
\label{lem_analytic_prop_1}
Let $(a_m)_{m=1}^\infty$ be a complex sequence.
If 
\begin{align}
  \frac{1}{n}\sum_{m=1}^n a_m = E +O(n^{-\delta})
\end{align}
for some $\delta>0$, then there exists a $K\in\R$ such that
\begin{align}
  \sum_{m=1}^n \frac{a_m}{m}
  &=
  E \log(n) + K + O(n^{-\delta}),
  \label{eq_part_sum_1}
\\
  \sum_{n/2< m \leq n} \frac{a_m}{m}
  &=
  E \log(2) + O(n^{-\delta}).
  \label{eq_part_sum_2}
%
\end{align}
\end{lemma}
%
\begin{proof}
The lemma follows from Lemma~\ref{eq_abel_summation} and straightforward verification.
%
\end{proof}
Finally we need
\begin{lemma}[H\"older inequality]
\label{lem_holder_inequality}
 Let $(a_m)_{m=1}^n$ and $(b_m)_{m=1}^n$ be finite sequences and  $p,q\geq 1$ such that $\frac{1}{p}+\frac{1}{q} =1$. Then
\begin{align}
  \sum_{m=1}^n |a_m b_m|
  \leq
  \left(  \sum_{m=1}^n |a_m|^p \right)^{1/p} \left(  \sum_{m=1}^n |b_m|^q \right)^{1/q}.
\end{align}
\end{lemma}

%

\subsection{Uniformly distributed sequences}
\label{sec_uniform_dist_seq}

We introduce in this section uniformly distributed sequences and state some interesting properties.
We follow the book~\cite{kuipers-niederreiter-74} and
omit most of the proofs since they are not difficult.

We begin with the definition of uniformly distributed sequences.
\begin{definition}
\label{def_d_uniform_dist_in_[a,b]}
Let $\ts=\left(t_m\right)_{m=1}^\infty$ be a sequence in the interval $[0,1]$. We set
\begin{align}
A_n([\alpha,\beta]) = A_n([\alpha,\beta],\ts) := \# \set{1\leq m \leq n: t_m\in [\alpha,\beta] }.
\end{align}
The sequence $\ts=\left(t_m\right)_{m\in\N}$ is called uniformly distributed in $[0,1]$ if
\begin{align}
\lim_{n\to\infty}
\left|\frac{A_n([\alpha,\beta])}{n}-(\beta-\alpha) \right|
=
0
\label{eq_unifrom_conv_intervall}
\end{align}
for each $\alpha,\beta$ with $0\leq \alpha \leq \beta \leq 1$.
\end{definition}
%
%
The following theorem shows that the name uniformly distributed is well chosen.
\begin{theorem}
\label{thm_equidist_integral_convergence}
Let $h: [0,1]\to \C$ be a proper Riemann-integrable function and $\ts = \left(t_m\right)_{m=1}^\infty$
be a uniformly distributed sequence, then

\begin{align}
\frac{1}{n} \sum_{m=1}^n h(t_m) \to \int_{0}^1 h(s)  \ ds.
\label{eq_thm_equidist_integral_convergence}
\end{align}

\end{theorem}
\begin{proof}
If $h(t) = \one_{[\alpha,\beta]}(t)$, then \eqref{eq_thm_equidist_integral_convergence} follows immediately from \eqref{eq_unifrom_conv_intervall}. For $h$ an arbitrary, proper Riemann-integrable function, one can uses an approximation argument.
\end{proof}

We would like to emphasize at that this point, that \eqref{eq_thm_equidist_integral_convergence} does not have to be true for improper Riemann integrable functions. One thus needs further assumptions to handle functions like $\log(t)$ or $t^{-1/2}$. This is indeed the cause why we need in Theorem~\ref{thm_main_result} the finite type condition.

%
We also need the discrepancy of a sequence.
%
\begin{definition}
\label{def_discrepcy}
Let $\left(t_m \right)_{m=1}^\infty$ be a sequence in $[0,1]$. We then define
\begin{align}
D_n &= D_n(\ts)
:=
\sup_{0 \leq \alpha \leq \beta \leq 1}
\left| \frac{A_n([\alpha,\beta])}{n}-(\beta-\alpha) \right|,
\label{eq_def_discrep}\\
D^*_n &= D^*_n(\ts)
:=
\sup_{0 \leq \beta \leq 1}
\left| \frac{A_n([0,\beta])}{n}-\beta \right|.
\label{eq_def_discrep_2}
\end{align}
We call $D_n$ the discrepancy and $D^*_n$ the $*$-discrepancy of the sequence $\ts$.
%
\end{definition}
%
%
%
It is easy to see that $D^*_n \leq D_n \leq 2 D^*_n$ and therefore $D_n$ and $D^*_n$ are equivalent.
We prefer to work with $D^*_n$ since we have a more explicit expression for it.
\begin{lemma}
\label{lem_exprssion_for_discrepcy}
Let $n$ be fixed and $\ts=\left(t_m\right)_{m\in\N}$ be a sequence in $[0,1]$ with $t_1 \leq\cdots \leq t_n$.
Then
\begin{align}
D_n^*(\ts)
=
\max_{1\leq m\leq n} \max\left(\left|t_m-\frac{m}{n}\right|,\left|t_m-\frac{m-1}{n}\right| \right)
\label{eq_discrepancy_simple_ex}
\end{align}
\end{lemma}
%
%
%

An important fact is that Theorem~\ref{thm_equidist_integral_convergence}, the discrepancy and uniformly distributed sequences are closely related. We have
\begin{lemma}
\label{lem_uniform_equiv_discrepancy}
Let $\ts=\left(t_m\right)_{m=1}^\infty$ be a sequence in $[0,1]$. The following conditions are equivalent
\begin{enumerate}
 \item $\ts$ is uniformly distributed,
 \item $\lim_{n\to\infty} D_n(\ts)=0$,
 \item Let $h:[0,1] \to \C$ be a proper Riemann-integrable function. Then
$$
\frac{1}{n} \sum_{m=1}^n h(t_m)  \to \int_0^1 h(s)\ ds \ \text{ for }n\to\infty.
$$
\end{enumerate}
\end{lemma}


We have introduced the discrepancy since it allows us to estimate the rate of convergence
in Theorem~\ref{thm_equidist_integral_convergence}. We have
\begin{theorem}[Koksma's inequality]
\label{thm_koksma_inequality}
Let $h: [0,1]\to \C$ be a function of bounded variation $V(h)$, and $\ts=\left(t_m\right)_{m\in\N}$ be a sequence in $[0,1]$. Then
\begin{align}
\left|
\frac{1}{n} \sum_{m=1}^n h(t_m) - \int_{0}^1 h(s)  \ ds
\right|
\leq
V(h)D_n^*(\ts).
\label{eq_thm_koksma_inequality}
\end{align}
%
\end{theorem}

%
We will work in Section~\ref{sec_poof_of_central_limit_for_Wn} with $h(s)=\log\bigl(f(e^{2\pi i s})$ for $f$ real analytic.
If $s_0$ is a zero of $f$, then
$$h(s) \sim K \log\bigl|s-s_0\bigr| \text{ and } \frac{d}{ds}h(s) \sim \frac{K}{|s-s_0|}$$
for $s\to s_0$ and $K$ a constant. We thus cannot apply Koksma's inequality in this situation. We instead use

%
\begin{theorem}
\label{thm_koksma_inequality_2}
Let $\delta>0$ and $ 0= s_0< s_1 < \cdots < s_{d+1} = 1$ with $s_{k+1}-s_k>2\delta$ be given.
We set
\begin{align}
  I:= \bigcup_{k=0}^d I_k \text{ with } I_k:=[ s_k+\delta, s_{k+1}-\delta].
\end{align}
Further let $\ts=(t_m)_{m=1}^n$ be a sequence in $I$ and $h:I \to \C$ be proper Riemann integrable function of bounded variation $V(h)$. Then
\begin{align}
\left|
 \frac{1}{n} \sum_{m=1}^n h(t_m) - \int_{I} h(s)  \ ds
\right|
\leq & \ D_n^*(\ts)V(h) +
\delta \left(\sum_{k=0}^d |h(s_k+\delta)|+ |h(s_{k+1}-\delta)|\right)
\label{eq_koksma_inequality_2}
\end{align}
with $V(h)$ calculated in $I$ and $D_n^*$ calculated in $[0,1]$.
\end{theorem}
\begin{proof}
We consider here only $I=[\delta,1-\delta]$. The general case is complete similar.
W.l.o.g we can assume that $t_1<t_2<\cdots<t_n$. We put $t_0:=\delta, t_{n+1}:= 1-\delta$ and look at
\begin{align}
 J:= \sum_{m=0}^n \int_{t_m}^{t_{m+1}} \left(s-\frac{m}{n}\right) \ dh(s) .
\end{align}
It follows immediately from Lemma~\ref{lem_exprssion_for_discrepcy} that $|J| \leq D_n^* V(h)$. This observation together with
a partial integration proves the theorem.
\end{proof}
We consider in Section~\ref{sec_poof_of_central_limit_for_Wn}
\begin{align}
  E_n:= \frac{1}{n} \sum_{m=1}^n \log\bigl(f(e^{2\pi i mt}) \text{ as } n\to \infty
\end{align}
for a fixed $t$ and $f$ real analytic. We are thus interested in the sequence $\left(\fracpart{m t}\right)_{m=1}^\infty$ with
\begin{align}
\fracpart{t}:=t-\intpart{t} \text{ and } \intpart{t}:=\max\set{n\in\Z, n\leq t}.
\label{def_ganz_anteil}
\end{align}
It is clear that we have to distinguish between $t$ rational and $t$ irrational. The case $t$ rational is easy to handle. The next lemma shows that we can apply Theorem~\ref{thm_koksma_inequality} and Theorem~\ref{thm_koksma_inequality_2} respectively for $t$ irrational.
\begin{lemma}
\label{lem_all_irrational_uniform_distributed}
Let $t\in\R$ be given.
The sequence $\left(\fracpart{m t}\right)_{m=1}^\infty$ is uniformly distributed in $[0,1]$ if and only if $t$ is irrational.
\end{lemma}
We can use Theorem~\ref{thm_koksma_inequality_2} only if we can estimate of the discrepancy of $\left(\fracpart{m t}\right)_{m=1}^\infty$ and find a $\delta = \delta(n)$ with $|\fracpart{m t}-s_k|>\delta$ for $1\leq m\leq n, 1 \leq k \leq d$ and $s_1,\dots s_d$ the zeros of $f(e^{2\pi i s})$ such that error terms in \eqref{eq_koksma_inequality_2} vanishes for $n\to \infty$. This is not so easy to do for arbitrary $t$ and $f$, but it can be done if $t$ is of finite type and all zeros of $f$ are roots of unity.

\begin{definition}
\label{def_finite_type}
Let $x=e^{2\pi i t}$ be given. We call $x$ and $t$ respectively of finite type if there exist constants $K,\gamma>0$ such that
\begin{align}
   \norm{\fracpart{mt}} > \frac{K}{m^\gamma} \ \text{ for all } m\in \Z\setminus\set{0}.
\label{eq_def_finite_type}
\end{align}
with $\norm{s}:= \min\set{s, 1-s}$.
We set
\begin{align}
 \eta = \inf \set{\gamma\in\R_{+}: \gamma \text{ fulfills \eqref{eq_def_finite_type}}}.
\label{eq_def_type}
\end{align}
The constant $\eta$ is called the type of $x$ and $t$ respectively.
%
\end{definition}
A simple computation now shows
\begin{lemma}
\label{lem_finite_type_approximation}
Let $t$ be of type $\eta$ and $q\in\N$ be fixed. For $\gamma>\eta$ there exist a constant $K$ such that
\begin{align}
   \max_{0 \leq p \leq q} \left|\fracpart{mt}-\frac{p}{q}\right| > \frac{K}{m^\gamma} \ \text{ for all } m\in \Z\setminus\set{0}.
\end{align}

\end{lemma}
If $t$ is of finite type, then the discrepancy of the sequence $\left(\fracpart{m t}\right)_{m\in\N}$ can be estimated with Erd\"os-Tur\'an-Koksma inequality. We get
\begin{theorem}
 \label{thm_finite_type_implies_small_discrepancy}
Let $t$ be of finite type $\eta$ and $\ts=\left(\fracpart{mt}\right)_{m=1}^\infty$.
We then have for each $\eps>0$
\begin{align}
 D_n(\ts) = O \left(n^{-\frac{1}{\eta}+\eps}\right)
  \label{eq_thm_finite_type_implies_small_discrepancy}
\end{align}
\end{theorem}
\begin{proof}
See \cite{kuipers-niederreiter-74}.
\end{proof}

\section{Auxiliary central limit theorem}
\label{sec_weak_limit_thm}
We prove in this section the following auxiliary central limit theorem
\begin{theorem}
\label{thm_normal_limit_dist}
Let $\theta >0$ be fixed.
Let $(c_m)_{m=1}^\infty$ be a sequence of complex numbers with $a_m=\Re(c_m)$, $b_m=\Im(c_m)$ and
\begin{enumerate}
\item \label{enum_thm_limit_dist_cond_1} $|b_m|\leq 2\pi$ and $|a_m| = O \bigl( \log(m) \bigr)$,
\item \label{enum_thm_limit_dist_cond_2} $\frac{1}{n} \sum_{m=1}^n |a_m| = E_a + O(n^{-\delta_a})$ for $n \to \infty$ and some $\delta_a>0$,
\item \label{enum_thm_limit_dist_cond_3} $\frac{1}{n} \sum_{m=1}^n |b_m| = E_b + O(n^{-\delta_b})$ for $n \to \infty$ and some $\delta_b>0$,
\item \label{enum_thm_limit_dist_cond_4} $\frac{1}{n} \sum_{m=1}^n a_m^2 \to V_a,\
					  \frac{1}{n} \sum_{m=1}^n b_m^2 \to V_b$, \ $\frac{1}{n} \sum_{m=1}^n a_m b_m \to E_{ab}$ for $n \to \infty$,
\item \label{enum_thm_limit_dist_cond_5} $\frac{1}{n} \sum_{m=1}^n |a_m|^3 = o\bigl(\log^{1/2}(n)\bigr)$,
\item \label{enum_thm_limit_dist_cond_6} It exists a $p > \frac{1}{\theta}$ with $\frac{1}{n} \sum_{n/2 < m < n} |a_m|^p = O(1)$.
\end{enumerate}
We define $A_n= A_n(\theta):=\sum_{m=1}^n a_m C_m^{(n)}$ and $B_n = B_n(\theta):=\sum_{m=1}^n b_m C_m^{(n)}$.

We then have
\begin{align}
\frac{A_n + i B_n-\Et{A_n + i B_n}}{\sqrt{\log(n)}}
	\xrightarrow{d}
 \mathcal{N}
\label{eq_thm_normal_limit_dist}
\end{align}
with $\mathcal{N}$ a complex normal distributed random variable with covariance matrix
\begin{align}
\Sigma = \theta\left(\begin{array}{cc}
            V_a & E_{ab} \\
            E_{ab} & V_b
          \end{array}\right) .
\end{align}
The real and the imaginary part of $\mathcal{N}$ are independent if and only if
$E_{ab}=0$.
\end{theorem}
%

Theorem~\ref{thm_normal_limit_dist} is similar to Lemma~3.1 in \cite{HKOS}, but there are two important differences:
\begin{itemize}
 \item We can calculate the covariance between the real and imaginary part and show when they are independent in the limit.
 \item We consider a more general measure on $S_n$.
\end{itemize}

%
\begin{proof}
We use in this proof the Feller-coupling (see Section~\ref{sec_cycle_counts}).
The random variables $C_m^{(n)}, C_m^{(n+1)}$  and $Y_m$ are therefore defined on the same space.
The idea of the proof is to replace $C_m^{(n)}$ by $Y_m$ and to do the computations with $Y_m$. For this reason we set
\begin{align}
\widetilde{A}_n:=\sum_{m=1}^n a_m Y_m \ \text{ and } \ \widetilde{B}_n:=\sum_{m=1}^n b_m Y_m.
\end{align}
We then can show
\begin{lemma}
\label{lem_help_asyp_behavior}
Let $a_m,b_m, A_n, B_n, \widetilde{A}_n$ and $\widetilde{B}_n$ be as above.
We then have
\begin{align}
\E{\left| \widetilde{A}_n + i \widetilde{B}_n  -A_n  - iB_n  \right|}
=
O(1) \text{ as } n\to \infty.
\label{eq_Cm_to_Ym=O(1)}
\end{align}
In particular, we see that
\begin{align}
  \frac{1}{\sqrt{\log(n)}}\bigl((A_n + i B_n) - (\widetilde{A}_n + i \widetilde{B}_n)\bigr)
\stackrel{d}{\longrightarrow}
0
\end{align}
and thus
$\frac{1}{\sqrt{\log(n)}}(A_n + i B_n)$
and
$\frac{1}{\sqrt{\log(n)}}(\widetilde{A}_n + i \widetilde{B}_n)$
have the same asymptotic behavior as $n\to \infty$.
\end{lemma}
We first finish the proof of Theorem~\ref{thm_normal_limit_dist} and then prove Lemma~\ref{lem_help_asyp_behavior}.

We do this by computing the characteristic function of
$\frac{1}{\sqrt{\log(n)}}(\widetilde{A}_n + i \widetilde{B}_n)$.
We set
\begin{align}
\chi^{(n)}
=
 \chi^{(n)}(t_a,t_b)
  =&
  \E{\exp\left(it_a  \frac{\widetilde{A}_n-\E{\widetilde{A}_n}}{\sqrt{\log(n)}}
   +it_b  \frac{\widetilde{B}_n-\E{\widetilde{B}_n}}{\sqrt{\log(n)}}\right)}.
   \label{eq_def_chi_n}
\end{align}
We can compute $\chi^{(n)}(t_a,t_b)$ explicitly since
$\E{e^{itY}} = \exp(\la (e^{it}-1))$
for $Y$ a Poisson distributed random variable with $\E{Y}=\la$.
We get
\begin{align}
\chi^{(n)}
&=
\E{\exp\left(  \frac{it_a\widetilde{A}_n}{\sqrt{\log(n)}}
   + \frac{it_b \widetilde{B}_n}{\sqrt{\log(n)}}\right)}
\exp\left(- \frac{it_a \E{\widetilde{A}_n}}{\sqrt{\log(n)}}
   -  \frac{it_b \E{\widetilde{B}_n}}{\sqrt{\log(n)}}\right)\nonumber\\
&=
\E{\exp\left( \sum_{m=1}^n \left(\frac{it_a a_m +it_n b_m}{\sqrt{\log(n)}}\right)Y_m
 \right)}
 \exp\left(- \frac{it_a \E{\widetilde{A}_n}}{\sqrt{\log(n)}}
   - \frac{ it_b \E{\widetilde{B}_n}}{\sqrt{\log(n)}}\right)\nonumber\\
&=
\exp{\left( \theta\sum_{m=1}^n\frac{e^{\frac{i(t_a a_m + it_b b_m )}{\sqrt{\log(n)}}}-1}{m}\right)}
\exp\left(-\frac{\theta}{\sqrt{\log(n)}} \left(\sum_{m=1}^n \frac{it_a a_m+it_b b_m}{m} \right) \right).
\label{eq_proof_limit_dist_char_fkt}
\end{align}
We now use that $\left| (e^{it} - 1)-\left(it-\frac{t^2}{2}\right) \right| \leq |t^3|$ and get
\begin{align}
e^{\frac{it_a a_m}{\sqrt{\log(n)}}}  e^{\frac{it_b b_m}{\sqrt{\log(n)}}}-1
=&
\frac{1}{\sqrt{\log(n)}}(it_a a_m+it_b b_m)\nonumber\\
&-\frac{1}{\log(n)} \left(\frac{1}{2} a_m^2 t_a^2
+ \frac{1}{2} b_m^2 t_b^2 + a_m b_m t_a t_b\right)\nonumber\\
&+
\frac{1}{\log^{3/2}(n)}O\left(|a_m|+|a_m|^2+|a_m|^3\right)
\label{eq_chi_asymp}
\end{align}
for $|t_a|,|t_b|\leq K$ with $K$ an arbitrary fixed number. We get
\begin{align}
\chi^{(n)}(t_a,t_b)
=
\exp&
\left(
-\frac{\theta}{2\log(n)}\sum_{m=1}^n \frac{a_m^2 t_a^2 + b_m^2 t_b^2}{m}
-
t_a t_b\frac{\theta}{\log(n)}\sum_{m=1}^n \frac{a_mb_m }{m}\right.
\nonumber\\
&+
\left.
\frac{\theta}{\log^{1/2}(n)} O\left(\frac{1}{\log(n)}\sum_{m=1}^n \frac{|a_m|+|a_m|^2+|a_m|^3}{m}\right)
\right).
\end{align}
We apply Lemma~\ref{lem_analytic_prop_1} to each summand.
The first summand converge by condition~\eqref{enum_thm_limit_dist_cond_4} to $-\theta(\frac{V_a}{2} t_a^2 + \frac{V_b}{2} t_b^2)$.
The second summand converge by condition~\eqref{enum_thm_limit_dist_cond_4} to $\theta t_a t_bE_{ab}$.
The third summand converge by
the conditions~\eqref{enum_thm_limit_dist_cond_2}, \eqref{enum_thm_limit_dist_cond_4} and \eqref{enum_thm_limit_dist_cond_5} to $0$.
Therefore
\begin{align}
 \chi_n(t_a,t_b)
\to
\exp\left(-\theta\frac{V_a}{2} t_a^2 -\theta \frac{V_b}{2} t_b^2- \theta E_{ab} t_a t_b\right)
\label{eq_fourier_limit_function}
\end{align}
pointwise for all $|t_a|,|t_b| \leq K$. Since $K$ was arbitrary, $\chi_n(t_a,t_b)$ converge everywhere.
This proves the theorem.
\end{proof}
We now finish the proof by proving Lemma~\ref{lem_help_asyp_behavior}.
\begin{proof}[Proof of Lemma~\ref{lem_help_asyp_behavior}]
It is enough to prove \eqref{eq_Cm_to_Ym=O(1)} since the other statements follow immediately with Markov's inequality and Slutsky's theorem.

We have to distinguish the case $\theta \geq 1$ and $0< \theta <1$. We begin with $\theta\geq 1$
and use \eqref{thm_feller_conv} and the conditions~\eqref{enum_thm_limit_dist_cond_2} and \eqref{enum_thm_limit_dist_cond_3}  to get
\begin{align}
 & \E{\left| \widetilde{A}_n + i \widetilde{B}_n  -A_n  - iB_n  \right|}
=	
\E{\left|\sum_{m=1}^n (a_n+ib_m)(Y_m-C_m^{(n)})  \right|} \nonumber\\
&\leq
 \left( \frac{\theta(\theta+1)}{\theta+n}\sum_{m=1}^n \bigl(|a_m|+|b_m|\bigr) \right)
 =
O(1).
\label{eq_thm_stong_replace_of_Cm}
\end{align}
This proves the lemma for $\theta \geq 1$.\\

The case $0<\theta <1$ is a little bit more difficult since \eqref{thm_feller_conv} is now weaker.
We solve this problem by splitting the sum.  We do this as follows
\begin{align}
\E{\left| \widetilde{A}_n + i \widetilde{B}_n  -A_n  - iB_n  \right|}
\leq \
&\E{\sum_{m \leq n/2}  \left|(a_n+ib_m)(Y_m-C_m^{(n)})  \right|}
\label{eq_change_var_theta<1}\\
+
&\E{\left| \sum_{n/2 < m \leq n} (a_m+ib_m)  Y_m\right|}\nonumber\\
+
&\E{\left| \sum_{n/2 < m \leq n} (a_m+ib_m)  C_m\right|}.\nonumber
\end{align}
We now show that each summand in \eqref{eq_change_var_theta<1} is $O(1)$.

Equation \eqref{thm_feller_conv} gives us
$\E{\left|C_m^{(n)}-Y_m\right|}\leq \frac{\theta(\theta+1)}{\theta+n-m} \leq \frac{2\theta(\theta+1)}{n}$
for $1\leq m \leq  n/2$.
We thus can use the same computation as in \eqref{eq_thm_stong_replace_of_Cm} to see
that the first summand in \eqref{eq_change_var_theta<1} is equal to
$O(1)$.

We next look at the second summand in \eqref{eq_change_var_theta<1}. We have
\begin{align}
\E{\left| \sum_{n/2 < m \leq n} (a_m+ib_m)  Y_m\right|}
&\leq
\sum_{n/2 < m \leq n}   \theta\left( \frac{|a_m|}{m} +\frac{|b_m|}{m} \right) .
\end{align}
We now use Lemma~\ref{lem_analytic_prop_1} and condition~\eqref{enum_thm_limit_dist_cond_2}
and \eqref{enum_thm_limit_dist_cond_3} to see that
\begin{align}
  \sum_{n/2 < m \leq n} \frac{|a_m|+|b_m|}{m}
  =
  (E_a + E_b) \log(2) + O(n^{-\min\set{\delta_a,\delta_b}}).
  \label{eq_asymptotic_sum_am_m}
\end{align}

This shows that the second summand in \eqref{eq_change_var_theta<1} is also $O(1)$.

We finally look at the third summand.

It is obvious that a permutation $\sigma\in S_n$ can have at most one cycle with length greater than $n/2$.
This fact and condition~\eqref{enum_thm_limit_dist_cond_1} together gives us
\begin{align}
\E{\left| \sum_{n/2 < m \leq n} (a_m+ib_m)  C_m\right|}
&=
\sum_{n/2 < m \leq n}
|a_m+ib_m| \ \Pb{C_m=1} \nonumber\\
&\leq
\sum_{n/2 < m \leq n}
(|a_m|+|b_m|) \ \Pb{C_m=1} \nonumber\\
&\leq
\sum_{n/2 < m \leq n} |a_m|\ \Pb{C_m=1} +2\pi \sum_{n/2 < m \leq n} \Pb{C_m=1}\nonumber\\
&=
\sum_{n/2 < m \leq n} |a_m|\ \Pb{C_m=1} +2\pi \Pb{\sum_{n/2 < m \leq n} C_m>0} \nonumber\\
&\leq
2\pi +\sum_{n/2 < m \leq n} |a_m|\ \Pb{C_m=1}.
\label{eq_upper_bound_sum_am_Cm}
\end{align}
We have used on the third line that $\set{C_{m_1} =1}\cap \set{C_{m_2} =1} =\emptyset$ for $m_1,m_2 >n/2$ and $m_1 \neq m_2$.\\
If the sequence $(a_m)_{m=1}^\infty$ is bounded by a constant $K$,  we can argue as in \eqref{eq_upper_bound_sum_am_Cm} to see
\begin{align}
 \E{\left| \sum_{n/2 < m \leq n} (a_m+ib_m)  C_m\right|}
\leq
2\pi+K = O(1).
\end{align}
%
%
%
If the sequence $(a_m)_{m=1}^\infty$ is unbounded, we have to be more careful.
We first look at $ \Pb{C_m=1}$ for $m > n/2$. It follows immediately from \eqref{eq_expec_Cm} that
\begin{align}
 \Pb{C^{(n)}_m=1}
=
\E{C^{(n)}_m}
=
\frac{\theta}{m} \frac{\binom{n-m-\gamma}{n-m}}{ \binom{n-\gamma}{n}}
\text{ with }\gamma = 1-\theta.
\end{align}
We now need an upper bound for $\frac{\binom{n-m-\gamma}{n-m}}{ \binom{n-\gamma}{n}}$.
A simple computation shows
\begin{align*}
 \frac{\binom{n-m-\gamma}{n-m}}{ \binom{n-\gamma}{n}}
=
\frac{n(n-1)\cdots (n-m+1)}{(n-\gamma)(n-\gamma-1)\cdots (n-m-\gamma+1)}.
\end{align*}
We thus have

\begin{align}
 \log \left( \frac{\binom{n-m-\gamma}{n-m}}{ \binom{n-\gamma}{n}}  \right)
&=
\sum_{k=n-m+1}^n \log(k) -\sum_{k=n-m+1}^n \log(k-\gamma)
=
\sum_{k=n-m+1}^n -\log\left(1-\frac{\gamma}{k}\right)\nonumber\\
&=
\sum_{k=n-m+1}^n \left(\frac{\gamma}{k} +O\left(\frac{\gamma^2}{k^2}\right) \right)
=
O(1) + \gamma \sum_{k=n-m+1}^n \frac{1}{k}.
\end{align}
We now use $\sum_{m=1}^n \frac{1}{m} = \log(n)+ K_1 + O\left(\frac{1}{n}\right)$ (see \cite[Theorem 3.2]{Apostol:84})
and distinguish the cases $m=n$ and $m<n$.
If $m=n$ then we get immediately
\begin{align}
 \log \left( \frac{\binom{n-m-\gamma}{n-m}}{ \binom{n-\gamma}{n}}  \right)
 \leq \gamma \log(n) + K_2.
\end{align}
If $m<n$ then
\begin{align}
 \log \left( \frac{\binom{n-m-\gamma}{n-m}}{ \binom{n-\gamma}{n}}  \right)
&=
O(1)+ \gamma \log(n)-\gamma \log(n-m) +O\left(\frac{1}{n} +\frac{1}{n-m}\right)\nonumber\\
&= -\gamma \log \left(1-\frac{m}{n} \right) +O(1)
\end{align}
since $n-m>0$.
We put everything together and get

\begin{align}
  \frac{ \binom{n-m-\gamma}{n-m}}{ \binom{n-\gamma}{n} }
\leq
\left\{
       \begin{array}{ll}
       K_3 (1-\frac{m}{n})^{-\gamma}, & \hbox{for }m< n,  \\
       K_4 n^\gamma , & \hbox{for } m=n.
\end{array}
\right.
\label{eq_upper_bound_binom}
\end{align}
We use condition~\eqref{enum_thm_limit_dist_cond_1}
and the H\"older inequality (see Lemma~\ref{lem_holder_inequality}) for some $p,q>1$, specified in a moment.
\begin{align}
&\sum_{n/2 < m \leq n} |a_m|\ \Pb{C_m=1}
=
{|a_n|} \Pb{C_n=1}
+
\sum_{n/2 < m < n} |a_m|\ \Pb{C_m=1} \nonumber\\
&\leq
\frac{O \bigl( \log(n) \bigr)}{n}n^\gamma
+
\frac{2}{n} \sum_{n/2 < m < n} |a_m|\ \frac{ \binom{n-m-\gamma}{n-m}}{ \binom{n-\gamma}{n} } \nonumber\\
&\leq
\frac{2}{n} \left( \sum_{n/2 < m < n} |a_m|^p \right)^{1/p} \left(  \sum_{n/2 < m < n} \left|  \frac{ \binom{n-m-\gamma}{n-m}}{ \binom{n-\gamma}{n} }   \right|^q \right)^{1/q} +O(1) \nonumber\\
&\leq
2K_3
\left(\frac{1}{n} \sum_{n/2 < m < n} |a_m|^p \right)^{1/p} \left( \frac{1}{n} \sum_{n/2 < m < n} \left(1-\frac{m}{n}\right)^{-\gamma q}\right)^{1/q} +O(1)
\end{align}
The second factor is a Riemann sum for $\int_{1/2}^1 (1-t)^{-\gamma q} \  dt$.
If we choose a $q>1$ with $\gamma q<1$ then the integral exists
and one can use Theorem~\ref{thm_koksma_inequality_2} to see that the second factor converge to this integral.
We now check if we can choose $q$ in a such a way that the product is $O \bigl(1)$. We have
\begin{align*}
  \gamma q<1
\ \Longleftrightarrow \
(1-\theta) < \frac{1}{q}
\ \Longleftrightarrow \
(1-\theta) < 1- \frac{1}{p}
\ \Longleftrightarrow \
\theta > \frac{1}{p}
\ \Longleftrightarrow \
p> \frac{1}{\theta}.
\end{align*}
Condition~\eqref{enum_thm_limit_dist_cond_6} now ensures the existence of a $p> \frac{1}{\theta}$ such that
$\frac{1}{n} \sum_{n/2 < m < n} |a_m|^p$ is $O(1)$.
We get with this $p$ (and $q$) that the product is bounded.
This shows that the third summand in \eqref{eq_change_var_theta<1} is $O (1)$. This prove the lemma and completes the proof of Theorem~\ref{thm_normal_limit_dist}.
\end{proof}
Many assumptions we need in the proof of Theorem~\ref{thm_normal_limit_dist} are to handle the case $\theta<1$.
If one is only interested in the case uniform measure ($\theta =1$) or in $\theta \geq 1$, one can weaken the assumptions. We state this as a corollary.
\begin{corollary}
\label{cor_thm_normal_limit_dist}
Let $\theta \geq 1$ be fixed.
Let $(c_m)_{m=1}^\infty$ be a sequence of complex numbers with $a_m=\Re(c_m)$, $b_m=\Im(c_m)$ and
\begin{enumerate}
\item[(1')] \label{enum_thm_limit_dist_cond_1'} $|b_m|\leq 2\pi$,
\item[(4)]  \label{enum_thm_limit_dist_cond_4'} $\frac{1}{n} \sum_{m=1}^n a_m^2 \to V_a,\
					  \frac{1}{n} \sum_{m=1}^n b_m^2 \to V_b$, \ $\frac{1}{n} \sum_{m=1}^n a_m b_m \to E_{ab}$ for $n \to \infty$,
\item[(5)] \label{enum_thm_limit_dist_cond_5'} $\frac{1}{n} \sum_{m=1}^n |a_m|^3 = o\bigl(\log^{1/2}(n)\bigr)$.
\end{enumerate}
Let $A_n$ and $B_n$ be as in Theorem~\ref{thm_normal_limit_dist}. We then have
\begin{align}
\frac{A_n + i B_n-\Et{A_n + i B_n}}{\sqrt{\log(n)}}
	\xrightarrow{d}
 \mathcal{N}
\label{eq_thm_normal_limit_dist_2}
\end{align}
with $\mathcal{N}$ a complex normal distributed random variable with covariance matrix
\begin{align}
\Sigma = \theta\left(\begin{array}{cc}
            V_a & E_{ab} \\
            E_{ab} & V_b
          \end{array}\right) .
\end{align}
The real and the imaginary part of $\mathcal{N}$ are independent if and only if
$E_{ab}=0$.

\end{corollary}

\section{Proof of the main Theorem~\ref{thm_main_result}}
\label{sec_poof_of_central_limit_for_Wn}
%
We now are ready to prove Theorem~\ref{thm_main_result}.
We recommend to read first Section~\ref{sec_uniform_dist_seq} before reading this proof.

\begin{proof}[Proof of Theorem~\ref{thm_main_result}]
We have by definition
\begin{align}
w^n(f)(x) = \sum_{m=1}^n C^{(n)}_m \log\bigl(f(x^m)\bigr).
\end{align}
We thus can apply Theorem~\ref{thm_normal_limit_dist} with $c_m:= \log\bigl(f(x^m)\bigr)$.
We now show  that the conditions~\eqref{enum_thm_limit_dist_cond_1} -- \eqref{enum_thm_limit_dist_cond_6} are fulfilled in all cases mentioned in Theorem~\ref{thm_main_result}.

We use the notation $x=e^{2\pi i t}, t_m:=\fracpart{mt}, \ts = \left(t_m\right)_{m=1}^\infty$ and define
\begin{align}
a(s)&:= \log\bigl| f(e^{2\pi i s}) \bigr|, &b(s):=& \arg\bigl( f(e^{2\pi i s})\bigr),\label{eq_def_a(s)_b(s)}\\
h(s)&:= \left|\log\bigl| f(e^{2\pi i s}) \bigr|\right|, & k(s):=& \left|\arg\bigl( f(e^{2\pi i s})\bigr) \right|.
\end{align}

%
%
%
%
%
%
%
\underline{Case 1.1}: $x$ not a root of unity and $f$ zero free.

Condition~\eqref{enum_thm_limit_dist_cond_1} is trivially fulfilled
since $a_m$ is bounded in this case and $b_m \leq 2\pi$ by definition of $w^n(f)$.
We next look at condition~\eqref{enum_thm_limit_dist_cond_2}. We have
\begin{align}
\frac{1}{n} \sum_{m=1} ^n |a_m|
=&
\frac{1}{n} \sum_{m=1} ^n |\Re(c_m)|
=
\frac{1}{n} \sum_{m=1} ^n  \Bigl|\log\bigl|f(x^m)\bigr| \Bigr|
=
\frac{1}{n} \sum_{m=1} ^n  h(t_m).
\end{align}

The function $h(s)$ is in this case a real analytic function on $[0,1]$
and we therefore can apply Theorem~\ref{thm_equidist_integral_convergence}
to see that the last expression converge to $\int_{0}^1 h(s) \ ds$.

We need in \eqref{enum_thm_limit_dist_cond_2} also the rate of convergence.
We thus use Theorem~\ref{thm_koksma_inequality} instead of Theorem~\ref{thm_equidist_integral_convergence}
and have therefore to estimate $D_n^*(\ts)$.
Since $x$ is assumed to be of finite type, see Definition~\ref{def_finite_type}, we can use Theorem~\ref{thm_finite_type_implies_small_discrepancy} and get $D_n(\ts) = O(n^{-\alpha})$ for some $\alpha >0$.
This gives the desired error rate.

It follows with the same argument that
\begin{align}
\frac{1}{n} \sum_{m=1} ^n |a_m|^p \to \int_{0}^1 h^p(s) \ ds \text{ for each } 1\leq p < \infty.
\end{align}
This shows that conditions~\eqref{enum_thm_limit_dist_cond_2}, \eqref{enum_thm_limit_dist_cond_5}, \eqref{enum_thm_limit_dist_cond_6}
and the first part of condition~\eqref{enum_thm_limit_dist_cond_4} are fulfilled.

We next look at $b_m =  \Im\left(\log\bigl(f(x^m)\bigr)\right) = \arg(f(x^m)) = b(t_m)$.
Since $f$ is real analytic, one can use function theory to show that there exists a finite set $D \subset [0,1]$ such that
$b(s)$ is real analytic in $[0,1] \setminus D$ and the limits
$ \lim_{s\uparrow s_0} \arg(f(e^{2\pi i s}))$ and  $\lim_{s\downarrow s_0} \arg(f(e^{2\pi i s}))$ exists for all $s_0 \in [0,1].$
We omit here the details since this is a standard argument. This shows that $b(s)$ is of bounded variation and that we can apply Theorem~\ref{thm_koksma_inequality} to $b(s)$. We get
\begin{align}
\frac{1}{n} \sum_{m=1}^n |b_m|
=
\int_{0}^1 \bigl|b(s)\bigr| \ ds
+ O(n^{-\delta_b})
\ \text{ and } \
\frac{1}{n} \sum_{m=1}^n b^2_m
\to
\int_{0}^1 \bigl(b(s)\bigr)^2 \ ds.
\end{align}
Similarly we get
\begin{align}
\frac{1}{n} \sum_{m=1}^n a_m b_m
\to
\int_{0}^1 a(s)b(s)\ ds
&=
\int_{0}^1 \Re\left(\log\bigl(f(e^{2 \pi i s})\bigr)\right) \Im\left(\log\bigl(f(e^{2 \pi i s})\bigr)\right) \ ds\nonumber\\
&=
\frac{1}{2}\Im\left( \int_{0}^1 \log^2\bigl(f(e^{2 \pi i s})\bigr) \ ds \right).
\end{align}
This gives the desired expression for the covariance mentioned in Theorem~\ref{thm_main_result}.\\
This shows that condition~\eqref{enum_thm_limit_dist_cond_3} and the rest of condition~\eqref{enum_thm_limit_dist_cond_4}
are fulfilled. This completes the proof of Theorem~\ref{thm_main_result} in this case. \\

\underline{Case 1.2}  $x$ not a root of unity and all zeros of $f$ are roots of unity.

The function $h(s)$ is in this case not anymore of bounded variation and we thus have to apply Theorem~\ref{thm_koksma_inequality_2}.

The discrepancy can be estimated as in case 1.1 as $D_n(\ts) = O(n^{-\alpha})$.

We define $s_1,\dots,s_d$ to be the zeros of $f(e^{2\pi is})$ and choose now a $\delta= \delta(n)$ such that the error terms in \eqref{eq_koksma_inequality_2} vanishes for $n\to \infty$.
By assumption, all zeros of $f$ are roots of unity and thus there exist a $q\in \N$ such that $s_k = \frac{p_k}{q}$ for some $p_k\in \N$. Since $x$ is assumed to be of finite type, we can apply Lemma~\ref{lem_finite_type_approximation} to see that
\begin{align}
  |t_m -s_k| \geq \frac{K}{n^\gamma} \text{ for } 1 \leq m \leq n, 1 \leq k \leq d
\end{align}
and some $\gamma,K>0$. We choose now $\delta = \frac{K}{n^\gamma}$.
Condition~\eqref{enum_thm_limit_dist_cond_1} follows immediately with this choice of $\delta$.

Let $s_0=0, s_{d+1}=1$ and let $V(h)$ be as in Theorem~\ref{thm_koksma_inequality_2}. We then get
\begin{align}
\left|\frac{1}{n} \sum_{m=1}^n h(t_m)- \int_{0}^{1} h(t)\ ds \right|
\leq&
\left|\int_0^\delta h(s) \ ds \right| +  \left|\int_{1-\delta}^1 h(s) \ ds \right|
+
\sum_{k=1}^d \left|\int_{s_k-\delta}^{s_k+\delta} h(s) \ ds \right|
\nonumber\\
&+
D_n^*(\ts) V(h)
+
\delta\sum_{k=0}^d
\bigl|h(s_k+\delta)\bigr| + \bigl|h(s_{k+1}-\delta)\bigr| \nonumber\\
\leq&
 O\left(n^{-\gamma} \log(n^\gamma)\right) + O\left(n^{-\alpha} \log(n^\gamma)\right)
\nonumber\\
=&
 O(n^{-\delta_a}) \text{ for some }\delta_a >0.
\label{eq_weak_convergence_caclulation}
\end{align}
Thus condition~\eqref{enum_thm_limit_dist_cond_2} is fulfilled.
A simple calculation shows that
\begin{align*}
h^p(s) \sim K_p \log^p(s), \qquad \frac{d}{ds} h^p(s) \sim pK_p  \frac{\log^{p-1}(s)}{s} \qquad \text{ for }  1\leq p <\infty.
\end{align*}
It is now easy to see that one can use the same argumentation as in \eqref{eq_weak_convergence_caclulation} also for $h^2, h^3$ and $h^p$.
The conditions~\eqref{enum_thm_limit_dist_cond_5}, \eqref{enum_thm_limit_dist_cond_6}
and the first part of condition~\eqref{enum_thm_limit_dist_cond_4}. The argumentation for $b_m$ is as above.

\underline{Case 2}: $x$ a root of unity of order $p$ and $f(x^m) \neq  0$ for all $1\leq m\leq p$.\\
We have
\begin{align}
\frac{1}{n} \sum_{m=1} ^n |a_m|
=&
\frac{1}{n} \sum_{m=1} h(t_m)
=
\sum_{k=1}^{p} \sum_{j=0}^{\intpart{\frac{n-k}{p}}} \frac{1}{jp+k} h(t_{jp+k}) \nonumber\\
=&
 \sum_{k=1}^{p}  h(t_{k}) \left( \frac{1}{\log(n)} \sum_{j=0}^{\intpart{\frac{n-k}{p}}}  \frac{1}{jp+k}   \right).
 \label{eq_weak_case_root_of_unity}
\end{align}
It is easy to see that
$$
\lim_{n\to\infty} \frac{1}{\log(n)} \sum_{j=0}^{\intpart{\frac{n-k}{p}}} \frac{1}{jp+k}
=
\lim_{n\to\infty} \frac{1}{\log(n/p)+\log(p)}\sum_{j=0}^{\intpart{\frac{n}{p}}} \frac{1}{jp}
=
\frac{1}{p}.
$$
This shows the desired convergence. To get the rate of convergence, one has to use
\begin{align}
  \sum_{j=1}^n \frac{1}{j} = \log(n) +K_{4} + O\left(\frac{1}{n}\right).
\label{eq_euler_summation}
\end{align}
This is a classical result and can be found for instance in the book \cite{Apostol:84}.
This shows that Condition~\eqref{enum_thm_limit_dist_cond_2} is fulfilled.
The other calculations are similar. We therefore omit them.

We have until now proven that
\begin{align}
 \frac{w^n(f) - \E{w^n(f)}}{\sqrt{\log(n)}} \stackrel{d}{\to} \mathcal{N}
\end{align}
in all cases mentioned in Theorem~\ref{thm_main_result} inclusive the calculation of the correlation.
To complete the proof, we have to show that
\begin{align}
  \frac{\E{w^n(f)}}{\sqrt{\log(n)}} - \theta\sqrt{\log(n)} m(f)\to 0.
\end{align}
We define as in the proof of Theorem~\ref{thm_normal_limit_dist}
\begin{align}
 \widetilde{w}^n(f)(x):= \sum_{m=1}^n \log\Bigl(f(x^m)\Bigr) Y_m.
\label{eq_def_w_tilde}
\end{align}
We know from Lemma~\ref{lem_help_asyp_behavior} that $w^n(f)$
and $\widetilde{w}^n(f)$ have the same asymptotic behavior. It is therefore enough to prove
\begin{align}
\frac{\E{\widetilde{w}^n(f)}}{\sqrt{\log(n)}} - \theta\sqrt{\log(n)} m(f)\to 0.
\end{align}
We first look at the case $x$ not a root of unity.
We get with Lemma~\ref{lem_analytic_prop_1} and \eqref{eq_weak_convergence_caclulation}
\begin{align}
 \E{\widetilde{w}^n(f)(x)}
 =
 \theta \sum_{m=1}^n \frac{\log\bigl(f(x^m)\bigr)}{m}
 =
 \theta \log(n) \int_{0}^1 \log\bigl(f(e^{2 \pi i s})\bigr) \ ds + O(1).
 \end{align}
We have by definition that $m(f)= \int_{0}^1 \log\bigl(f(e^{2 \pi i s})\bigr) \ ds$ and thus
\begin{align}
 \frac{1}{\sqrt{\log(n)}}
 \sum_{m=1}^n \frac{\log\bigl(f(z^m)\bigr)}{m}
=&
\sqrt{\log(n)} m(f) + O\left(\frac{1}{\sqrt{\log(n)}}\right).
\label{eq_dist_m(f)_to_E}
\end{align}
In $x$ is a root of unity, one has to replace \eqref{eq_weak_convergence_caclulation} by \eqref{eq_euler_summation}.
We omit the details since this calculations are similarly.
\end{proof} 
\section{Estimation of the Wasserstein distance}
\label{sec_wasserstein}

We estimate in this section the convergence rate of the random variable $w^n(f)(x)$ with respect to the Wasserstein distance,
see Theorem~\ref{thm_estimate_wasserstein_1_dim} and Theorem~\ref{thm_estimate_wasserstein_2_dim}.
Unfortunately we have to distinguish between complex and real valued random variables.
We thus look in Section~\ref{sec_wasser_real_and_im_sep} first at the real and the imaginary part of $w^n(f)$ separately and then look in Section~\ref{sec_wass_2dim} at the complex case.
\subsection{The real and the imaginary part separately}
\label{sec_wasser_real_and_im_sep}
\begin{definition}
\label{def_wasserstein_1_dim}
 Let $X_1,X_2$ be real valued random variables not necessarily defined on the same space.
 The Wasserstein distance $\mathrm{d}_{W}$ between $X_1$ and $X_2$ is then defined as
 \begin{align}
 \mathrm{d}_{W}(X_1,X_2)
 :=
 \sup_{g\in \mathcal{G} }
 \left| \E{g(X_1)} - \E{g(X_2)} \right|
 \end{align}
 with
 \begin{align}
   \mathcal{G}
   :=
   \set{g\in C^1(\R,\R) :  \sup_{t\in \R} |g'(t)| \leq 1 }.
 \end{align}
 \end{definition}

It is easy to see that $\mathrm{d}_{W}$ is a metric. We now show
\begin{theorem}
\label{thm_estimate_wasserstein_1_dim}
Let $\theta>0$, and $x$ be either not a root of unity and of finite type  or a root of unity.
Let $\mathcal{N} = \mathcal{N}_a +i \mathcal{N}_b$ be as in Theorem~\ref{thm_main_result}, then
\begin{align}
\mathrm{d}_{W} \left( \mathcal{N}_a, \
\Re \left(\frac{w^n(f)(x)}{\sqrt{\log(n)}}-\theta\sqrt{\log(n)} m(f)(x) \right)
\right)
=
O\bigl(\log^{-1/2}(n)\bigr),\\
\mathrm{d}_{W} \left( \mathcal{N}_b, \
\Im \left(\frac{w^n(f)(x)}{\sqrt{\log(n)}}-\theta\sqrt{\log(n)} m(f)(x) \right)
\right)
=
O\bigl(\log^{-1/2}(n)\bigr).
\end{align}
\end{theorem}
We prove this theorem by reducing it to
\begin{theorem}[{\cite[Theorems~3.1~and~3.2]{MR2235448}}]
\label{thm_essitmate_Wasser_with_stein}
Let $\xi_1,\cdots, \xi_n$ be independent random variables with
\begin{align}
 \E{\xi_m}=0,\   \E{\sum_{m=1}^n \xi^2_m } = V \text{ and } \sum_{m=1}^n \E{|\xi^3_m|} < \infty.
\end{align}
Then
\begin{align}
 d_W \left( \sum_{m=1}^n \xi_m  , \ \mathcal{N}(0,V)  \right)
 \leq
 3 V^{3/2}\sum_{m=1}^n \E{|\xi^3_m|}.
\end{align}
\end{theorem}

\begin{proof}[Proof of Theorem~\ref{thm_estimate_wasserstein_1_dim}]
We use the notation
\begin{align}
  c_m  &=\log\bigl(f(x^m)\bigr), \  a_m =\Re(c_m), \ b_m =\Re(c_m),\\
  a(t) &=\log\bigl|f(e^{2\pi i t})\bigr|, \ b(t) =\arg \bigl(f(e^{2\pi i t})\bigr).
  \label{eq_notationen_wasser}
\end{align}
We use as in the proof of Theorem~\ref{thm_main_result} the Feller coupling, see Section~\ref{sec_cycle_counts}.
The random variables $w^n(f)(x)$ and $\widetilde{w}^n(f)(x)$ are thus defined on the same space.
We now get with Lemma~\ref{lem_help_asyp_behavior}
\begin{align}
  d_W &\left(\Re \left(\frac{w^n(f)(x)}{\sqrt{\log(n)}}\right) , \Re \left(\frac{\widetilde{w}^n(f)(x)}{\sqrt{\log(n)}} \right) \right)\nonumber\\
&=
  \sup_{g\in \mathcal{G} }
 \left| \E{g\left(\Re \left(\frac{w^n(f)(x)}{\sqrt{\log(n)}}\right)\right)} - \E{g\left(\Re \left(\frac{\widetilde{w}^n(f)(x)}{\sqrt{\log(n)}}\right)\right)}
  \right| \nonumber\\
&=
  \sup_{g\in \mathcal{G} }
 \left| \E{g\left(\Re \left(\frac{w^n(f)(x)}{\sqrt{\log(n)}}\right)\right)
 -
 g\left(\Re \left(\frac{\widetilde{w}^n(f)(x)}{\sqrt{\log(n)}}\right)\right)}
  \right| \nonumber\\
&\leq
 \E{\left|
\frac{w^n(f)(x)}{\sqrt{\log(n)}} -
 \frac{\widetilde{w}^n(f)(x)}{\sqrt{\log(n)}}
\right|}
=
O\bigl(\log^{-1/2}(n)\bigr).
\label{eq_wasser_change_w_wtilde}
\end{align}
Similarly for the imaginary part.

This shows that we can replace $w^n(f)$ by $\widetilde{w}^n(f)$ in Theorem~\ref{thm_estimate_wasserstein_1_dim}.
It follows with the triangle inequality and \eqref{eq_dist_m(f)_to_E} that one can replace also
\begin{align}
 \theta\sqrt{\log(n)} m(f)(x)
\ \text{ by } \
   \frac{\theta}{\sqrt{\log(n)}} \sum_{m=1}^{n} \frac{c_m}{m}.
\end{align}
We now have
\begin{align}
  \Re\left( \widetilde{w}^n(f) - \frac{\theta}{\sqrt{\log(n)}} \sum_{m=1}^{n} \frac{c_m}{m} \right)
  =
  \frac{1}{\sqrt{\log(n)}} \sum_{m=1}^n a_m \left( Y_m-\frac{\theta}{m} \right).
\end{align}
Since all $Y_m$ are independent, we can apply Theorem~\ref{thm_essitmate_Wasser_with_stein} for
\begin{align}
  \xi_m = \frac{1}{\sqrt{\log(n)}} a_m \left( Y_m-\frac{\theta}{m} \right).
\end{align}
We only have to check that the assumptions of Theorem~\ref{thm_essitmate_Wasser_with_stein} are fulfilled. We use the computations in Section~\ref{sec_poof_of_central_limit_for_Wn} and get
\begin{align}
\E{\sum_{m=1}^n \xi^2_m }
&=
\frac{\theta}{\log(n)} \sum_{m=1}^n \frac{a_m^2}{m} = V_a + O(n^{-\delta_a}) \label{eq_approx_Va},
\\
\sum_{m=1}^n \E{|\xi^3_m|}
&=
O\left(
\frac{1}{\log^{3/2}(n)} \sum_{m=1}^n \frac{|a_m|^3}{m} \right)
=
O\left(\frac{1}{\sqrt{\log(n)}}\right).
\end{align}
An application of the triangle inequality shows that one can neglect the error term $O(n^{-\delta_a})$ in \eqref{eq_approx_Va}. This proves the theorem.
\end{proof}

\subsection{The two dimensional case}
\label{sec_wass_2dim}
The definition of the Wasserstein distance can be extended to $\R^d$ without any problems. Unfortunately it is often very difficult to handle the case $d>1$. In many situations it is much easier to take stronger assumptions on the test functions $g$. We thus set
\begin{definition}
\label{def_wasserstein_2_dim}
 Let $X_1,X_2$ be random variables with values in $\R^d$, not necessarily defined on the same space.
 The weak Wasserstein distance $\mathrm{d}_{wW}$ between $X_1$ and $X_2$ is then defined as
 \begin{align}
 \mathrm{d}_{wW}(X_1,X_2)
 :=
 \sup_{g\in \mathcal{G} }
 \left| \E{g(X_1)} - \E{g(X_2)} \right|
 \end{align}
 with
 \begin{align}
   \mathcal{G}
   :=
   \set{g\in C^\infty(\R^d,\R) ;  M_1(g) \leq 1 \text{ and } M_2(g) \leq 1}
 \end{align}
and
\begin{align}
   M_k(g)
   :=
   \sup_{
   \substack{
   u\in\R^d\\
   1\leq i_1,\cdots, i_k \leq d
   }}
   \left|
   \frac{\partial^k}{\partial u_{i_1} \cdots \partial u_{i_k}} g(u)
   \right|.
\end{align}

 \end{definition}
We now show
\begin{theorem}
\label{thm_estimate_wasserstein_2_dim}
Let $\theta>0$, and $x$ be either not a root of unity and of finite type  or a root of unity.
Let $\mathcal{N}$ be as in Theorem~\ref{thm_main_result}. We then have
\begin{align}
\mathrm{d}_{wW} \left( \mathcal{N}, \
\frac{w^n(f)(x)}{\sqrt{\log(n)}}-\theta\sqrt{\log(n)} m(f)(x)
\right)
=
O\bigl(\log^{-1/2}(n)\bigr).
\end{align}
\end{theorem}
This theorem is an alternative proof of Theorem~\ref{thm_main_result} since the weak Wasserstein distance is a metric on the space of random variables.

We prove Theorem~\ref{thm_estimate_wasserstein_2_dim} with Stein's method. We can not give here a full introduction, but try to illustrate at least the idea of Stein's method.
Assume that a random variable $Z$ and a $g \in \mathcal{G}$ are given. In many situations  one can find a ``good" function $\widetilde{g}$ only depending on $g$ and $Z$ such that
\begin{align}
  \E{g(X)} - \E{g(Z)}
  =
  \E{\widetilde{g}(X)}.
  \label{eq_stein_equation_illustration}
\end{align}
for all random variables $X$. This simplifies the study of the Wasserstein distance since one now has to consider only one random variable. This reformulation is of course only useful if we can find a $\widetilde{g}$ with good properties. This is surprisingly often the case.
For $d=1$, $N$ a standard normal distributed random variable and $g\in \C^1( \R,\R)$, one has
\begin{align}
  \E{f_g^\prime(X) - X f_g (X)}
  =
  \E{g(X)} - \E{g(N)}
\end{align}
with
\begin{align}
  f_g(x):= e^{x^2/2} \int_{-\infty}^x (g(t)-\E{N})e^{-t^2/2} \ dt.
\end{align}
To proof Theorem~\ref{def_wasserstein_2_dim}, we need of course a $2$-dimensional version. We use here
\begin{lemma}[{\cite[Lemma 1]{multivarite_normal}}]
  \label{thm_steins_equality_2_dim}
Let $g:\R^2 \to \R$ be a smooth function and $\mathcal{N}$ be a bivariate normal distributed random variable with covariance matrix $\Sigma$.
We set
\begin{align}
  Ug(x)
  :=
  \int_0^1 \frac{1}{2t} \left( \E{g(\sqrt{t}x + \sqrt{1-t}N_\Sigma)} - \E{g(N_\Sigma)}     \right) \ dt.
  \label{eq_def_Ug}
\end{align}
We then have
\begin{align}
  \E{g(X)} - \E{g(\mathcal{N})}
  =
  \E{ \scalar{X,\nabla (Ug)(X)}  - \scalarHS{ \mathrm{Hess(Ug)(X),\Sigma} } }
  \label{eq_stein_eq_bivariate}
\end{align}
with $\scalar{\cdot,\cdot}$ the standard inner product on $\R^2$ and $\scalarHS{M_1,M_2} = \mathrm{Tr}(M_1 M^T_2)$.
\end{lemma}
Before we proof Theorem~\ref{thm_estimate_wasserstein_2_dim}, we have to take a look at the derivations of $Ug$.
\begin{lemma}[{\cite[lemma 2]{multivarite_normal}}]
\label{lem_Upper bounds_Ug}
Let $g: \R^d \to \R$ be a smooth function and $Ug$ as in Lemma~\ref{thm_steins_equality_2_dim}.
If $\Sigma$ is positive definite, then
\begin{align}
   M_3(Ug)
  &\leq
  K\cdot
  M_2(g)
\end{align}
with $K$ only depending on $\Sigma$.
\end{lemma}
There exists also upper bounds for $M_k(Ug)$ if $\Sigma$ is non negative definite, but we do not need them here.
The reason is the following lemma.
\begin{lemma}
\label{lem_sigma_non_negativ_definit}
The covariance matrix $\Sigma$ in Theorem~\ref{thm_main_result} is singular if and only if there exist $\gamma =(\gamma_a,\gamma_b)\in\R\setminus\set{0}$ such that $ \gamma_a\log|f(x^m)| =  \gamma_b \arg \big(f(x^m)\bigr)$ for all $m \in \N$ except finitely many.
\end{lemma}
\begin{proof}
We prove this lemma only for $x$ not a root of unity and of finite type. The case $x$ a root of unity is similarly.
We define $a(s)$ and $b(s)$ as \eqref{eq_notationen_wasser}.
We know from Section~\ref{sec_poof_of_central_limit_for_Wn} that
\begin{align}
  V_a=\int_0^1 a^2(s) \ ds, \  E_{ab}=\int_0^1 a(s)b(s) \ ds,\  V_b=\int_0^1 b^2(s) \ ds.
\end{align}
%
%
Since $\Sigma$ is a $2 \times 2 $ matrix, one can directly compute the eigenvalues. One gets after a small calculation that $\Sigma$ is non-negative definite if and only if
\begin{align}
  \left(\int_0^1 a^2(s) \ ds \right)  \left(   \int_0^1 b^2(s) \ dt \right)
  \geq
  \left(  \int_0^1 a(s)b(s) \ dt  \right)^2.
  \label{eq_wasser_sigma_positiv_schwarz}
\end{align}
and $\Sigma$ is singular if and only if we have equality in \eqref{eq_wasser_sigma_positiv_schwarz}.
But equation \eqref{eq_wasser_sigma_positiv_schwarz} is the Schwarz inequality for $L^2$.
This shows that $\Sigma$ is always non negative definite and that $\Sigma$ is singular if and only if the functions $a(t)$ and $b(t)$ are linearly dependent.
This proves the lemma since $a(t)$ and $b(t)$ have only finitely many discontinuity points and $x^{m_1} \neq x^{m_2} $ for $m_1 \neq m_2$.
\end{proof}

We are now ready to prove Theorem~\ref{thm_estimate_wasserstein_2_dim}.
\begin{proof}[Proof of Theorem~\ref{thm_estimate_wasserstein_2_dim}]
Let $g \in \mathcal{G}$ be given with $M_1(g)\leq 1, M_2(g) \leq 1$.
We have to distinguish the cases $\Sigma$ singular and $\Sigma$ regular.

We start with the singular case. We know from Lemma~\ref{lem_sigma_non_negativ_definit} that $\Sigma$ is singular if and only if $a_m \neq b_m$ for only finitely many $m$.
We thus have
\begin{align*}
  \E{g \left(\frac{1}{\sqrt{\log(n)}} \sum_{m=1}^n c_m Y_m    \right) }
  =
  \E{\widehat{g} \left(\frac{1}{\sqrt{\log(n)}} \sum_{m=1}^n a_m Y_m    \right)}
  +
  O\bigl (\log^{-1/2}(n)\bigr)
\end{align*}
with $ \widehat{g}(t):= g\bigl( (1+i)t \bigr)$.
This shows that we can argue as in the one-dimensional case. \\

We now come to $\Sigma$ regular.\\
One can use the same argumentation as in the proof of Lemma~\ref{thm_estimate_wasserstein_1_dim} to see that it is enough to show
\begin{align}
\mathrm{d}_{wW} \left( \mathcal{N}_{\widetilde{\Sigma}}, \
\frac{1}{\sqrt{\log(n)}}\sum_{m=1}^n c_m \left( Y_m-\frac{1}{m} \right)
\right)
\end{align}

with $\mathcal{N}_{\widetilde{\Sigma}}$ a bivariate normal distribution with covariance matrix
\begin{align}
\widetilde{\Sigma}
=
\left(
  \begin{array}{cc}
    \widetilde{V_a} & \widetilde{E}_{ab} \\
    \widetilde{E}_{ab} & \widetilde{V_b} \\
  \end{array}
\right)
=
\left(
  \begin{array}{cc}
    \sum_{m=1}^n \frac{\widetilde{a}^2_m}{m}    & \sum_{m=1}^n \frac{\widetilde{a}_m \widetilde{b}_m}{m} \\
    \sum_{m=1}^n \frac{\widetilde{a}_m \widetilde{b}_m}{m}  & \sum_{m=1}^n \frac{\widetilde{b}^2_m}{m} \\
  \end{array}
\right)
\end{align}
and
\begin{align}
\widetilde{a}_m
= \frac{a_m}{\sqrt{\log(n)}}, \
\widetilde{b}_m= \frac{b_m}{\sqrt{\log(n)}}
\ \text{ and } \
\widetilde{c}_m= \frac{c_m}{\sqrt{\log(n)}}.
\end{align}
We now use \eqref{eq_stein_eq_bivariate} to give the desired upper bound.
Let $Ug$ be as in \eqref{eq_def_Ug}. We use the notation
\begin{align}
  \nabla (Ug)
  =
  \binom{g_a}{g_b}, \
  \mathrm{Hess}(Ug)
  =
  \left(
    \begin{array}{cc}
      g_{aa} & g_{ab} \\
      g_{ab} & g_{bb}
    \end{array}
  \right)
   \ \text{ and } \
   X
 :=
 \sum_{ m = 1}^n \widetilde{c}_m \left( Y_m-\frac{1}{m} \right).
\end{align}
We also introduce
\begin{align*}
 X_a
 :=
 \sum_{ m = 1}^n \widetilde{a}_m \left( Y_m-\frac{1}{m} \right), \
 X_b
 =
 \sum_{m=1}^n \widetilde{b}_m \left( Y_k-\frac{1}{k} \right),\
 X_m
 &=
 \sum_{\substack{k \neq m \\ 1\leq k \leq n}} \widetilde{c}_k  \left( Y_k-\frac{1}{k} \right).
\end{align*}
We now identify $\C$ with $\R^2$ via $a+ib = \binom{a}{b}$.
We first look at the summand $\E{\scalar{X, \nabla (Ug)(X)} } = \E{X_a g_a(X)} + \E{X_b g_b(X)}$. We have
\begin{align}
   \E{X_a g_a(X)}
   &=
   \sum_{m=1}^n \widetilde{a}_m \E{ \left( Y_m-\frac{1}{m} \right) g_a(X)}\nonumber\\
   &=
   \sum_{m=1}^n \widetilde{a}_m \E{ \left( Y_m-\frac{1}{m} \right) \bigl( g_a(X) - g_a(X_m) \bigr)}\nonumber\\
   &=
   \sum_{m=1}^n \widetilde{a}_m \E{ \left( Y_m-\frac{1}{m} \right) \int_0^{Y_m}  \scalar{ \nabla g_a(X_m + t \widetilde{c}_m), \binom{\widetilde{a}_m}{\widetilde{b}_m} } \ dt  } \nonumber\\
   &=
   \sum_{m=1}^n \widetilde{a}_m \E{\left( Y_m-\frac{1}{m} \right)  \int_0^{\infty}  \scalar{ \nabla g_a(X_m + t \widetilde{c}_m), \binom{\widetilde{a}_m}{\widetilde{b}_m} } \mathbf{1}_{\set{0\leq t \leq Y_m}} \ dt  }\nonumber\\
   &=
   \sum_{m=1}^n \widetilde{a}_m \E{\int_0^{\infty} \scalar{ \nabla g_a(X_m + t \widetilde{c}_m), \binom{\widetilde{a}_m}{\widetilde{b}_m} } K_m(t) \ dt }\nonumber\\
   &=
    \sum_{m=1}^n \E{\int_0^\infty  \Bigl( \widetilde{a}^2_m g_{aa}(X_m + t \widetilde{c}_m) + \widetilde{a}_m \widetilde{b}_m g_{ab}(X_m + t \widetilde{c}_m) \Bigr) K_m(t) \ dt }
   \label{eq_wasser_mit_Km_1}
\end{align}
with
\begin{align}
  K_m(t)
  =
  \E{\left( Y_m-\frac{1}{m} \right)   \mathbf{1}_{\set{0\leq t \leq Y_m}}     }.
\end{align}
We have used for the second equality that $\E{Y_m-1/m}=0$ and that $X_m$ is independent of $Y_m$.
We have of course also to justify the existence of the integrals, but this follows immediately from $M_1(g) \leq 1, M_2(g)\leq 1$ and Lemma~\ref{lem_Upper bounds_Ug}. \\
We next look at
\begin{align}
\E{\scalarHS{ \mathrm{Hess}(Ug)(X),\Sigma }} = \widetilde{V}_a \E{g_{aa}(X)} +2 \widetilde{E}_{ab} \E{g_{ab}(X)} +\widetilde{V}_b \E{g_{bb}(X)}.
\end{align}
A direct computation shows that
\begin{align}
  \int_0^\infty K_m(t) \ dt = \frac{1}{m}
  \ \text{ and } \
  \int_0^\infty t K_m(t) \ dt = \frac{1}{m} + O\left( \frac{1}{m^2} \right)
  \label{eq_wasser_computation_of_km}
\end{align}
with $O(\cdot)$ independent of $n$ and $\Sigma$.
We thus have
\begin{align}
  \E{\widetilde{V}_a g_{aa}(X)}
  =
  \E{ g_{aa}(X) \sum_{m=1}^n \widetilde{a}^2_m \frac{1}{m} }
  =
  \sum_{m=1}^n \E{\widetilde{a}^2_m \int_0^\infty g_{aa}(X)K_m(t) \ dt}.
  \label{eq_wasser_mit_Km_2}
\end{align}
We combine \eqref{eq_wasser_mit_Km_1} and \eqref{eq_wasser_mit_Km_2} and get
\begin{align}
&\E{\scalar{X, \nabla (Ug)(X)} }
-
\E{\scalarHS{ \mathrm{Hess}(Ug)(X),\Sigma }}\nonumber\\
=&
\sum_{m=1}^n \E{\int_0^\infty \widetilde{a}_m^2 \Bigl(g_{aa}(X_m+\widetilde{c}_m t)-g_{aa}(X_m + \widetilde{c}_m Y_m) \Bigr) K_m(t)\ dt } \nonumber\\
&+
\sum_{m=1}^n 2\E{\int_0^\infty \widetilde{a}_m \widetilde{b}_m \Bigl(g_{ab}(X_m+\widetilde{c}_m t)-g_{ab}(X_m + \widetilde{c}_m Y_m) \Bigr) K_m(t)\ dt } \nonumber\\
&+
\sum_{m=1}^n \E{\int_0^\infty \widetilde{b}^2_m \Bigl(g_{bb}(X_m+\widetilde{c}_m t)-g_{bb}(X_m + \widetilde{c}_m Y_m) \Bigr) K_m(t)\ dt }.
\end{align}
We now use \eqref{eq_wasser_computation_of_km} and Lemma~\ref{lem_Upper bounds_Ug} to get
\begin{align}
\left|
\E{\int_0^\infty \Bigl(g_{aa}(X_m+\widetilde{c}_m t)-g_{aa}(X_m + \widetilde{c}_m Y_m) \Bigr) K_m(t) \ dt }
\right| \nonumber\\
\leq \
\E{\int_0^\infty M_3(Ug) |\widetilde{c}_m|(t + Y_m) K_m(t) \ dt }
\nonumber\\
\leq \
M_2(g)|\widetilde{c}_m| \int_0^\infty (t + \frac{1}{m}) K_m(t) \ dt \nonumber\\
\leq \
\frac{K}{\sqrt{\log(n)}} M_2(g) \frac{|a_m|+|b_m|}{m}.
\end{align}
Thus
\begin{align}
\E{\scalar{X, \nabla (Ug)(X)} }
-
\E{\scalarHS{ \mathrm{Hess}(Ug)(X),\Sigma }}\\
\leq
\frac{\breve{K}}{\log^{3/2}(n)}
\left(\sum_{m=1}^n |a_m^3|+|a_m^2b_m|+|a_m b^2_m|+|b_m^3|
\right).\nonumber
\end{align}
It follows with a computation similar to computation in \eqref{eq_weak_convergence_caclulation}
that the last expression is $O\bigl(\log^{-1/2}(n)\bigr)$.
This proves the theorem.
\end{proof}

%

\section*{Acknowledgements}
I would like to thank Andrew Barbour for some helpful discussions.

\bibliographystyle{plain}
\bibliography{literatur}

\begin{thebibliography}{10}

\bibitem{Apostol:84}
T.~Apostol.
\newblock {\em Introduction to analytic number theory}.
\newblock Springer-Verlag, New York, 1984.

\bibitem{MR1177897}
R.~Arratia, A.~D. Barbour, and S.~Tavar{\'e}.
\newblock Poisson process approximations for the {E}wens sampling formula.
\newblock {\em Ann. Appl. Probab.}, 2(3):519--535, 1992.

\bibitem{barbour}
R.~Arratia, A.D. Barbour, and S.~Tavar{\'e}.
\newblock {\em Logarithmic combinatorial structures: a probabilistic approach}.
\newblock EMS Monographs in Mathematics. European Mathematical Society (EMS),
  Z{\"u}rich, 2003.

\bibitem{bump}
D.~Bump.
\newblock {\em Lie groups}, volume 225 of {\em Graduate Texts in Mathematics}.
\newblock Springer-Verlag, New York, 2004.

\bibitem{MR2235448}
L.~Chen and Q.~Shao.
\newblock Stein's method for normal approximation.
\newblock In {\em An introduction to {S}tein's method}, volume~4 of {\em Lect.
  Notes Ser. Inst. Math. Sci. Natl. Univ. Singap.}, pages 1--59. Singapore
  Univ. Press, Singapore, 2005.

\bibitem{PhysRevLett.75.69}
O.~Costin and J.L. Lebowitz.
\newblock Gaussian fluctuation in random matrices.
\newblock {\em Phys. Rev. Lett.}, 75(1):69--72, Jul 1995.

\bibitem{associated_class}
P.~Dehaye and D.~Zeindler.
\newblock On averages of randomized class functions on the symmetric groups and
  their asymptotics.

\bibitem{MR0325177}
W.~J. Ewens.
\newblock The sampling theory of selectively neutral alleles.
\newblock {\em Theoret. Population Biology}, 3:87--112; erratum, ibid. 3
  (1972), 240; erratum, ibid. 3 (1972), 376, 1972.

\bibitem{HKOS}
B.M Hambly, P.~Keevash, N.~O'Connell, and D.~Stark.
\newblock The characteristic polynomial of a random permutation matrix.
\newblock {\em Stochastic Process. Appl.}, 90(2):335--346, 2000.

\bibitem{snaith}
J.P. Keating and N.C. Snaith.
\newblock {Random matrix theory and $\zeta(1/2+it)$}.
\newblock {\em Commun. Math. Phys.}, 214:57--89, 2000.

\bibitem{kuipers-niederreiter-74}
L.~Kuipers and H.~Niederreiter.
\newblock {\em Uniform Distribution of Sequences}.
\newblock Wiley, New-York, 1974.

\bibitem{multivarite_normal}
E.~Meckes.
\newblock On stein's method for multivariate normal distribution.
\newblock {\em IMS Collections}, 5:153--178, 2009.

\bibitem{EJP2010-34}
D.~Zeindler.
\newblock Permutation matrices and the moments of their characteristics
  polynomials.
\newblock {\em Electronic Journal of Probability}, 15:1092--1118, 2010.

\end{thebibliography}

\end{document}